\documentclass[10pt]{amsart}
\usepackage{graphicx}
\usepackage{amscd}
\usepackage{amsmath}
\usepackage{amsthm}
\usepackage{amsfonts}
\usepackage{amssymb}
\usepackage{mathrsfs}
\usepackage{bm}
\usepackage{enumerate}
\usepackage{amsrefs}
\usepackage{xcolor}
\usepackage[colorlinks, citecolor=blue, linkcolor=red, pdfstartview=FitB]{hyperref}

\usepackage{marginnote}	
\usepackage{soul} 			

\newcommand{\bB}{{\mathbb{B}}}
\newcommand{\bC}{{\mathbb{C}}}
\newcommand{\bD}{{\mathbb{D}}}

\newcommand{\bM}{{\mathbb{M}}}
\newcommand{\bN}{{\mathbb{N}}}

\newcommand{\bS}{{\mathbb{S}}}
\newcommand{\bT}{{\mathbb{T}}}

  \newcommand{\A}{{\mathcal{A}}}
  \newcommand{\B}{{\mathcal{B}}}

  \newcommand{\E}{{\mathcal{E}}}
  \newcommand{\F}{{\mathcal{F}}}
  
\renewcommand{\H}{{\mathcal{H}}}

  \newcommand{\K}{{\mathcal{K}}}

\renewcommand{\O}{{\mathcal{O}}}

\renewcommand{\S}{{\mathcal{S}}}

  \newcommand{\X}{{\mathcal{X}}}

\newcommand{\fF}{{\mathfrak{F}}}

\newcommand{\fJ}{{\mathfrak{J}}}
\newcommand{\fK}{{\mathfrak{K}}}

\newcommand{\fL}{{\mathfrak{L}}}
\newcommand{\fl}{{\mathfrak{l}}}

\newcommand{\fs}{{\mathfrak{s}}}
\newcommand{\fT}{{\mathfrak{T}}}

\newcommand{\sB}{{\mathscr{B}}}
\newcommand{\rA}{\mathrm{A}}

\newcommand{\rC}{\mathrm{C}}

\newcommand{\eps}{\varepsilon}
\renewcommand{\phi}{\varphi}

\newcommand{\upchi}{{\raise.35ex\hbox{$\chi$}}}

\newcommand{\ol}{\overline}

\newcommand{\qand}{\quad\text{and}\quad}

\newcommand{\re}{\operatorname{Re}}
\newcommand{\spn}{\operatorname{span}}

\newcommand{\spec}{\operatorname{Spec}}

\newcommand{\Prim}{\operatorname{Prim}}

\newtheorem{lemma}{Lemma}[section]
\newtheorem{theorem}[lemma]{Theorem}
\newtheorem{proposition}[lemma]{Proposition}
\newtheorem{corollary}[lemma]{Corollary}
\newtheorem{theoremx}{Theorem}

\theoremstyle{definition}

\newtheorem{example}{Example}

\author{Rapha\"el Clou\^atre}

\address{Department of Mathematics, University of Manitoba, Winnipeg, Manitoba, Canada R3T 2N2}

\email{raphael.clouatre@umanitoba.ca\vspace{-2ex}}
\author{Ian Thompson}
\email{thompsoi@myumanitoba.ca\vspace{-2ex}}
\thanks{R.C. was partially supported by an NSERC Discovery Grant. I.T. was partially supported by an NSERC CGS-M Scholarship and University of Manitoba Graduate Fellowship.}

\title{Minimal boundaries for operator algebras}
\begin{document}
\begin{abstract}
We study \emph{boundaries} for unital operator algebras. These are sets of irreducible $*$-representations that completely capture the spatial norm attainment for a given subalgebra. Classically, the Choquet boundary is the minimal boundary of a function algebra and it coincides with the collection of peak points. We investigate the question of minimality for the non-commutative counterpart of the Choquet boundary and show that minimality is equivalent to what we call the Bishop property. Not every operator algebra has the Bishop property, but we exhibit classes of examples that do. Throughout our analysis, we exploit various non-commutative notions of peak points for an operator algebra. When specialized to the setting of $\rC^*$-algebras, our techniques allow us to provide a new proof of a recent characterization of those $\rC^*$-algebras admitting only finite-dimensional irreducible representations. 
\end{abstract}
\maketitle

\section{Introduction}\label{S:intro}

Let $X$ be a compact metric space and let $\A\subset \rC(X)$ be a unital norm-closed subalgebra that separates the points of $X$. In other words, $\A$ is a \emph{uniform algebra}. A \emph{boundary} for $\A$ is a subset of $X$ on which every function in $\A$ attains its maximum modulus. The problem of identifying the unique smallest boundary $M\subset X$ for $\A$ is a classical question with a very elegant solution, which we briefly recall. 

Consider the set $C\subset X$ of points $x$ for which the character on $\A$ of evaluation at $x$ has a unique positive linear extension to $\rC(X)$ -- this is called the \emph{Choquet boundary} of $\A$. A seminal theorem of Choquet \cite[Chapter 3]{phelps2001} implies that $C$ is in fact a boundary for $\A$. Therefore, $M\subset C$. 

Next, we recall that a \emph{peak point} for $\A$ is a point $x\in X$ for which there is a function $a\in \A$ such that $a(x)=1>|a(y)|$ if $y\neq x$. It is immediate that if $P\subset X$ denotes the set of peak points, then $P\subset M$. In turn, a classical theorem of Bishop shows that $P=C$ \cite{bishop1959bdry},\cite[Corollary 8.2]{phelps2001}; so that the Choquet boundary of $\A$ is the minimal boundary.  Closed boundaries are also of interest, and the smallest such is the \emph{Shilov boundary}. It follows from the above discussion that the Shilov boundary is simply the closure of the Choquet boundary.

In this paper, we will be concerned with non-commutative counterparts to the aforementioned objects. A norm-closed subalgebra $\A\subset B(\H)$ will be referred to as an \emph{operator algebra} and for the present discussion we will assume that $\A$ is unital. In \cite{arveson1969}, Arveson initiated an ambitious program through which $\A$ is analyzed by means of non-commutative analogues of tools from uniform algebra theory. From his work there emerged the ubiquitous notion of a \emph{boundary representation} for $\A$ -- an irreducible $*$-representation $\beta$ of $\rC^*(\A)$ that is the unique completely contractive extension of $\beta|_\A$. The collection of these boundary representations is usually thought of as the Choquet boundary of $\A$. Indeed, the \emph{$\rC^*$-envelope} $\rC^*_e(\A)$ is the smallest $\rC^*$-algebra generated by a completely isometric copy of $\A$ \cite{hamana1979}, and it is an important fact that the $\rC^*$-envelope can be identified via boundary representations \cite{MS1998},\cite{dritschel2005},\cite{arveson2008},\cite{DK2015}. Roughly speaking, this reflects the aforementioned fact that the Shilov boundary of a uniform algebra is the closure of its Choquet boundary. 

On account of the successful realization of Arveson's program for constructing the $\rC^*$-envelope, one may wish to import other tools from classical uniform algebra theory into the non-commutative operator algebraic world. Peak points are natural candidates for this endeavour. Accordingly, different types of non-commutative peak points have been considered and successfully exploited in recent years, see for instance \cite{hay2007},\cite{BHN2008},\cite{arveson2010},\cite{arveson2011},\cite{BR2011},\cite{BR2013},\cite{blecher2013},\cite{clouatre2018lochyp},\cite{CTh2020fdim},\cite{DP2022},\cite{CT2022ncHenkin}. The overarching goal of our work here is to utilize non-commutative peak points to determine whether or not the Choquet boundary of a unital operator algebra is its minimal boundary. This is made more precise below.

Towards this end, in Section \ref{S:ncpeak}, we focus on three types of non-commutative peak points: representations, states and projections. In the commutative world, all of these objects simply encode the usual notion, but the non-commutative situation is much richer and will necessitate the use of all of these incarnations of peak points. 
For the purposes of the current discussion, it suffices to say that an irreducible $*$-representation $\pi$ of $\rC^*(\A)$ is a  \emph{local $\A$-peak representation} 
if there is an integer $n\geq 1$ and an element $B\in \bM_n(\A)$ with $\|B\|=1$ such that
\[
1>\|P_F \sigma^{(n)}(B)|_F\|
\]
for every irreducible $*$-representation $\sigma:\rC^*(\A)\to B(\H_\sigma)$ unitarily inequivalent to $\pi$ and every finite-dimensional subspace $F\subset \H^{(n)}_\sigma$. 
Next, a state $\omega$ on $\rC^*(\A)$ is an \emph{$\A$-peak state} if there is $a\in \A$ such that $\omega(a^*a)>\phi(a^*a)$ for every other state $\phi$. Finally, we say that a projection $q\in \rC^*(\A)^{**}$ is an \emph{$\A$-peak projection} if there is a contraction $a\in \A$ such that $aq=q$ and   for every unital completely positive map $\psi:\rC^*(\A)\to B(\H_\psi)$ with $\H_\psi$ finite-dimensional and $\|\psi(a^*a)\|=1$, we must have $\|\psi(q)\|=1$. Theorem \ref{T:peakdiff} and the surrounding discussion exhibit connections between different notions of non-commutative peak points.

In Section \ref{S:peakChoquet}, we explore whether non-commutative peak points necessarily belong to the Choquet boundary. This question is readily settled for representations and states: see Theorem \ref{T:locpeakbdry}, Example \ref{E:peakstateuep}, and Proposition \ref{P:statepeakbdry}. This issue is more subtle for projections. Using recent developments in matrix convexity \cite{HL2021}, we obtain the following (see Theorem  \ref{T:peakprojker}).

\begin{theoremx}\label{T:A}
Let $\A\subset B(\H)$ be a unital operator algebra. Let $\pi$ be a unital $*$-representation  of $\rC^*(\A)$ whose support projection is an $\A$-peak projection. Then, there is a set $\sB_\pi$ of boundary representations for $\A$ on $\rC^*(\A)$ such that if we set $\fJ_\pi=\bigcap_{\beta\in \sB_\pi}\ker \beta$, then 
\[
\A\cap \fJ_\pi\subset \ker \pi\subset \fJ_\pi \qand \rC^*(\A\cap \fJ_\pi)=\ker \pi.
\]
\end{theoremx}

We show in Example \ref{E:singleton} that this result is sharp, in the sense that the set $\sB_\pi$  cannot generally be taken be a singleton (or even finite), and that the conclusion cannot be improved to $\fJ_\pi=\ker \pi$. In particular, there are irreducible $*$-representations whose support projection is an $\A$-peak projection, yet the representations do not lie in the Choquet boundary of $\A$.

In Section \ref{S:Bishop}, we are interested in the reverse problem, namely whether a boundary representation is automatically a non-commutative peak point. More generally, we seek non-commutative versions of the theorem of Bishop referred to above. 
We say that $\A$ has the \emph{Bishop property} in $\rC^*(\A)$ whenever all boundary representations for $\A$ are local $\A$-peak representations. We show in Corollary \ref{C:C*Bishop} that any unital $\rC^*$-algebra has the Bishop property in itself. Example \ref{E:nonBishop} exhibits an example of a unital operator algebra that fails to have the Bishop property, thereby illustrating that the most naive non-commutative generalization of Bishop's theorem does not hold true. Nevertheless, a rich source of well-studied and important algebras with this property comes from function theoretic operator theory, in the form of multiplier algebras of unitarily invariant reproducing kernel Hilbert spaces on the unit ball. In Theorem \ref{T:AHBishop}, we establish the following.

\begin{theoremx}\label{T:B}
Let $\fF$ be a regular, unitarily invariant, complete Pick space on the unit ball. Let $\rA(\fF)\subset B(\fF)$ denote the norm-closure of the polynomial multipliers. Then, $\rA(\fF)$ has the Bishop property inside of $\rC^*(\rA(\fF))$.
\end{theoremx}

Examples of spaces to which this result applies include the Hardy space on the unit disc, the Dirichlet space on the unit disc, as well as the Drury--Arveson space on the unit ball. These spaces and their multiplier algebras have generated much research interest in recent years; see for instance \cite{CD2016duality},\cite{CH2018},\cite{CT2021},\cite{DH2020} and references therein. We also show in Theorem \ref{T:tensorBishop} that, under some regularity conditions, the Bishop property is preserved by tensor products, thus allowing us to produce more examples (Corollary \ref{C:tensorBishop}).

As mentioned previously, the main driving force of this paper is to determine whether the Choquet boundary of an operator algebra is a minimal boundary. Accordingly, we call a set $\Delta$ of irreducible $*$-representations of $\rC^*(\A)$ a \emph{boundary} for $\A$ if for every integer $n\geq 1$ and for every $B\in \bM_n(\A)$ there is an element of $\Delta$, say $\pi:\rC^*(\A)\to B(\H_\pi)$, and a finite-dimensional subspace $F\subset \H_\pi^{(n)}$ such that 
\[
\|P_F \pi^{(n)}(B)|_F\|=\|B\|.
\]
In Theorem \ref{T:minbdry}, we show that the Choquet boundary of an operator algebra is indeed a minimal boundary, at least when the Bishop property is satisfied. 

\begin{theoremx}\label{T:C}
The following statements are equivalent.
\begin{enumerate}[{\rm (i)}]
\item The algebra $\A$ has the Bishop property inside of $\rC^*(\A)$.
\item A set $\Delta$ of irreducible $*$-representations of $\rC^*(\A)$ is a boundary for $\A$ if and only if every boundary representation for $\A$ is unitarily equivalent to some element of $\Delta$.
\end{enumerate}
\end{theoremx}

Finally, in Section \ref{S:bdryC*} we study boundaries in the setting of $\rC^*$-algebras. Since the previously defined notion is trivial in this case (Corollary \ref{C:minbdryC*}), we introduce two variations. Let $\fT$ be a $\rC^*$-algebra and let $\Delta$ be a collection of irreducible $*$-representations of $\fT$.  We say that  $\Delta$ is a \emph{maximizing  boundary} for $\fT$ if for every $t\in \fT$ and every irreducible $*$-representation $\pi$, we can find $\delta\in \Delta$ such that 
$
\|\delta(t)\|\geq \|\pi(t)\|.
$
We say that $\Delta$ is a \emph{minimizing boundary} for $\fT$ if
for every $t\in \fT$ and every irreducible $*$-representation $\pi$, we can find $\delta\in \Delta$ such that 
$
\|\delta(t)\|\leq \|\pi(t)\|.
$
In the separable case, these boundaries can be characterized as follows (Theorems \ref{T:maxbdryC*} and \ref{T:minbdryC*})

\begin{theoremx}\label{T:D}
Assume that $\fT$ is separable. Then, the following statements hold.
\begin{enumerate}[{\rm (1)}]
\item The set  $\Delta$ is a maximizing boundary for $\fT$ if and only if the kernel of any irreducible $*$-representation of $\fT$ contains the kernel of some element of $\Delta$. 
\item The set  $\Delta$ is a minimizing boundary for $\fT$ if and only if the kernel of any irreducible $*$-representation of $\fT$ is contained in the kernel of some element of $\Delta$. \end{enumerate}
\end{theoremx}

As an application, in Corollary \ref{C:CS} we recover and provide a new proof of a recent result of Courtney and Shulman \cite[Theorem 4.4]{CS2019} that characterizes $\rC^*$-algebras for which all irreducible $*$-representations are finite-dimensional.

\section{Non-commutative peak points}\label{S:ncpeak}

In this section, we define the different types of non-commutative peak points that will be utilized throughout the paper, and establish new connections between them. As mentioned in the introduction, our analogues of peak points will come in three different varieties: states, representations and projections.

We will start with states. The following elementary fact will help motivate our definition.

\begin{lemma}\label{L:peakstatepure}
Let $\fT$ be a unital $\rC^*$-algebra and let $\omega$ be a state. Assume that there is a positive element $t\in \fT$ such that $\omega(t)>\psi(t)$ for every pure state $\psi$ on $\fT$ distinct from $\omega$. Then, $\omega$ is pure and $\omega(t)>\phi(t)$ for any state $\phi$ on $\fT$ distinct from $\omega$.
\end{lemma}
\begin{proof}
Consider the set $\E$ of states $\phi$ on $\fT$ such that $\phi(t)=\|t\|$.  It is readily verified that $\E$ is a non-empty, weak-$*$ closed face of the state space of $\fT$. The assumption on $\omega$ implies that $\E$ can have at most one extreme point, namely $\omega$ itself. In turn, an application of the Krein--Milman theorem yields that $\E=\{\omega\}$, and both statements follow from this.
\end{proof}

Let $\A\subset B(\H)$ be a unital operator algebra and let $\omega$ be a state on $\rC^*(\A)$. We say that $\omega$ is an \emph{$\A$--peak state} if there is a contraction $a\in \A$ such that
$
\omega(a^*a)>\phi(a^*a)
$
for every pure state $\phi$ on $\rC^*(\A)$ with $\phi\neq \omega$. By Lemma \ref{L:peakstatepure}, it then follows that $\omega$ is necessarily pure, that $\omega(a^*a)=\|a\|^2$, and that $\omega(a^*a)>\phi(a^*a)$ for any state $\phi$ on $\fT$ distinct from $\omega$.  This notion of peak state  differs from that found in \cite{clouatre2018lochyp}, where operator systems (rather than algebras) were the focal point. No use of this other version will be made in our paper, so no confusion should arise.

Before introducing our next notion, we record another elementary fact.

\begin{lemma}\label{L:locpeakrep}
Let $\fT$ be a unital $\rC^*$-algebra, let $\pi:\fT\to B(\H_\pi)$ be an irreducible $*$-representation and let $t\in \fT$. Then, the following statements are equivalent.
\begin{enumerate}[{\rm (i)}]
\item We have 
$
\|P_G \sigma(t)|_G\|<\|t\|
$
for every irreducible $*$-representation $\sigma:\fT\to B(\H_\sigma)$  unitarily inequivalent to $\pi$ and every finite-dimensional subspace $G\subset \H_\sigma$. 

\item   There is a finite-dimensional subspace $F\subset \H_\pi$ such that
\[
\|t\|=\|P_F\pi(t)|_F\|>\|P_G \sigma(t)|_G\|
\]
for every irreducible $*$-representation $\sigma:\fT\to B(\H_\sigma)$ unitarily inequivalent to $\pi$ and every finite-dimensional subspace $G\subset \H_\sigma$. 
\end{enumerate}
\end{lemma}
\begin{proof}
Assume that (i) holds. Let $\psi$ be a pure state on $\fT$ with the property that $\psi(t^*t)=\|t\|^2$, and let $\rho:\fT\to B(\H_\rho)$ denote its GNS representation with distinguished cyclic unit vector $\xi\in \H_\rho$. It is well known that $\rho$ is irreducible since $\psi$ is pure.  Let $X=\spn\{\xi,\rho(t)\xi\}\subset \H_\rho$. Then,
\[
\|P_X \rho(t)|_X\|\geq \|\rho(t)\xi\|=\psi(t^*t)^{1/2}=\|t\|.
\]
By assumption, this forces $\rho$ to be unitarily equivalent to $\pi$, so that (ii) holds.
\end{proof}

If $\A\subset B(\H)$ is an operator algebra, then for every positive integer $n\geq 1$, we then denote by $\bM_n(\A)\subset B(\H^{(n)})$ the operator algebra consisting of $n\times n$ matrices with entries in $\A$. Here and below, $\H^{(n)}=\H\oplus \H\oplus \ldots \oplus \H$. Given a linear map $\phi:\A\to B(\H_\phi)$, we denote by $\phi^{(n)}:\bM_n(\A)\to B(\H^{(n)}_\phi)$ its natural ampliation.

We now return to our discussion of non-commutative peak points, by recalling a notion introduced in \cite{CTh2020fdim}. We say that an irreducible $*$-representation $\pi$ of $\rC^*(\A)$ is a \emph{local $\A$-peak representation}  if there is $n\geq 1$ and $B\in \bM_n(\A)$ with $\|B\|=1$ such that
\[
1>\|P_G \sigma^{(n)}(B)|_G\|
\]
for every irreducible $*$-representation $\sigma:\rC^*(\A)\to B(\H_\sigma)$ unitarily inequivalent to $\pi$ and every finite-dimensional subspace $G\subset \H^{(n)}_\sigma$. As seen in Lemma \ref{L:locpeakrep}, this automatically implies that there is a finite-dimensional subspace $F\subset \H^{(n)}_\pi$ such that $1=\|P_F\pi^{(n)}(B)|_F\|$. 

The notion of a local $\A$-peak representation is genuinely weaker than that of an $\A$-peak representation introduced in \cite{arveson2010,arveson2011}: an example can be found in \cite[Example 2]{CTh2020fdim}. Alternatively, this can also be observed by choosing $\A$ to be the Cuntz $\rC^*$-algebra $\O_2$, which is known to be simple, and hence to admit no $\O_2$-peak representations at all. On the other hand, local $\O_2$-peak representations necessarily exist by virtue of Theorem \ref{T:peakdiff} below.

Before we can state this result, we introduce yet another version of non-commuta\-tive peak points. For this purpose, recall that the second dual of any $\rC^*$-algebra $\fT$ can be given the structure of a von Neumann algebra containing $\fT$ as a weak-$*$ dense $\rC^*$-subalgebra. Further, given a $*$-representation $\pi:\fT\to B(\H_\pi)$, the map $\pi^{**}:\fT^{**}\to B(\H_\pi)^{**}$ is a weak-$*$ continuous $*$-representation. Since $B(\H_\pi)$ is a dual space, there is also a weak-$*$ continuous $*$-representation $\widehat\pi:\fT^{**}\to B(\H_\pi)$ extending $\pi$. The reader should consult  \cite[Section A.5]{BLM2004} for more detail on these topics.

We say that a projection $q\in \rC^*(\A)^{**}$ is an \emph{$\A$-peak projection} if there is a contraction $a\in \A$ such that $aq=q$ and   for every unital completely positive map $\psi:\rC^*(\A)\to B(\H_\psi)$ with $\H_\psi$ finite-dimensional and $\|\psi(a^*a)\|=1$, we must have $\|\psi(q)\|=1$. 
In this case, we say that $a$ \emph{peaks} at $q$. 
The concept of a peak projection was first introduced in \cite{hay2007}, although the definition used there is (superficially) weaker than the one we gave above. We explain below how the two definitions actually coincide, at least for separable algebras. 

Peak projections can be characterized very neatly by a non-commutative analogue of a classical theorem of Glicksberg \cite[Theorem II.12.7]{gamelin1969}. Such a characterization requires the use of Akemann's non-commutative topology \cite{akemann1969,akemann1970}.
A projection $q\in \rC^*(\A)^{**}$ is \emph{closed} if there is a net $(t_i)$ of contractions in $\rC^*(\A)$ that decreases to $q$ in the weak-$*$ topology of $ \rC^*(\A)^{**}$. Equivalently, $q$ is closed precisely when there is a closed left ideal $\fJ\subset \rC^*(\A)$ such that $\fJ^{\perp\perp}=\rC^*(\A)^{**}(I-q)$.

\begin{theorem}\label{T:ncGlick}
Let $\A\subset B(\H)$ be a separable unital operator algebra, and let $q\in \rC^*(\A)^{**}$ be a closed projection. Then, $q$ is an $\A$-peak projection if and only if $q\in \A^{\perp\perp}$.
\end{theorem}
\begin{proof}
This can be found in \cite[Theorem 3.2 and Corollary 3.3]{CTh2020fdim}.
\end{proof}

In particular, the previous theorem combined with \cite[Proposition 5.6]{hay2007} shows that, at least in the separable setting, our notion of $\A$-peak projection coincides with that first introduced in that paper.

\subsection{Connections between the various types of peak points}

In this subsection, we aim to exhibit some relationships between the various types of non-commutative peak points introduced above. 
We start with a known fact that shows how to go from peak projections to local peak representations. If $\pi$ is a $*$-representation of some $\rC^*$-algebra $\fT$, then its \emph{support projection} is the unique central projection $\fs_\pi\in \fT^{**}$ with the property that $(\ker \pi)^{\perp\perp}=\fT^{**}(I-\fs_\pi)$. By definition, this projection is closed.

\begin{theorem}\label{T:locpeak}
Let $\A\subset B(\H)$ be a  unital operator algebra and let $\pi$ be a finite-dimensional irreducible $*$-representation of $\rC^*(\A)$. If $\fs_\pi$ is an $\A$-peak projection, then $\pi$ is a local $\A$-peak representation.
\end{theorem}
\begin{proof}
This is contained in the proof of \cite[Theorem 3.4]{CTh2020fdim}.
\end{proof}

As shown in the discussion following \cite[Theorem 3.4]{CTh2020fdim}, the converse of the previous theorem fails. Furthermore, the finite-dimensionality assumption on $\pi$ cannot be removed, as we illustrate next.

\begin{example}\label{E:peakAD}
Let $H^2$ denote the Hardy space on the open unit disc $\bD\subset \bC$, namely the Hilbert space of analytic functions on $\bD$ with square-summable Taylor coefficients at the origin. The unilateral shift operator can be represented on this space via the isometric operator $S$ of multiplication by the variable $z$. Then, $\rC^*(S)\subset B(H^2)$ is the Toeplitz algebra $\fT_1$. Let $\A_1\subset \fT_1$ denote the unital norm-closed subalgebra generated by $S$. It is well known that $\fT_1$ contains the ideal $\fK$ of compact operators on $H^2$, and that there is a $*$-isomorphism $\fT_1/\fK\cong \rC(\bT)$ sending $S+\fK$ to $z$ (here and throughout, $\bT$ denote the unit circle). It follows from von Neumann's inequality that the quotient map $\fT_1\to\fT_1/\fK$ is completely isometric on $\A_1$, whence the identity representation of $\fT_1$ is not a boundary representation for $\A_1$ by virtue of Arveson's boundary theorem \cite[Theorem 2.1.1]{arveson1972}. Invoking \cite[Theorem 3.1]{CTh2020fdim}, we see that the identity representation cannot be a local $\A_1$-peak representation. Finally, the support projection of the identity representation is simply the unit, which clearly lies in $\A_1^{\perp\perp}$, so that it is an $\A_1$-peak projection by Theorem \ref{T:ncGlick}.
\qed
\end{example}

In light of the previous example, one may ask if \emph{boundary} representations of arbitrary dimension whose support projection is an $\A$-peak projection must necessarily be a local $\A$-peak representation. At present, we do not know the answer to this question. Drawing inspiration from Example \ref{E:peakAD}, one may guess that a source of possible counterexamples would be the natural multivariate generalizations of $\A_1\subset \fT_1$ considered in Section \ref{S:Bishop}. Unfortunately, as we shall see in Theorem \ref{T:AHBishop}, all boundary representations are local $\A$-peak representations in that setting. Alternatively, the Cuntz algebras $\O_n$ (with $n\geq 2$) also seem like a promising source of counterexamples. Indeed, since they are simple $\rC^*$-algebras, the support projection of an irreducible $*$-representation lies in any unital subalgebra. While much is known about the finer structure of the $*$-representations of $\O_n$, we are not aware of any results related to peaking behaviour.

Our next task is to see how peak states relate to the other kinds of peak points. For this purpose, we introduce the following terminology. Let $\fT$ be a unital $\rC^*$-algebra and let $\phi$ be a state on $\fT$. Then, the \emph{left support projection} of $\phi$ is the unique projection $\fl_\phi\in \fT^{**}$ such that
\[
\{x\in \fT^{**}:\phi(x^*x)=0\} =\fT^{**}(I-\fl_\phi).
\]
A basic property that will be required below is the following standard fact.

\begin{lemma}\label{L:supppure}
Let $\fT$ be a unital $\rC^*$-algebra and let $\psi$ be a pure state on $\fT$. If $\phi$ is an arbitrary state on $\fT$, then $\psi=\phi$ if and only if $\phi(\fl_\psi)=1$. 
\end{lemma}
\begin{proof}
See \cite[Lemma 2.2]{clouatre2018lochyp}.
\end{proof}

It is relevant to note that the left support projection of a pure state is automatically closed.  This is known to experts (see \cite[Proposition 3.13.6]{pedersen1979book}), but we provide additional details for the reader's convenience.

\begin{lemma}\label{L:purestateopen}
Let $\fT$ be a unital $\rC^*$-algebra and let $\psi$ be a pure state on $\fT$. Then, the left support projecton $\fl_\psi$ is a closed projection in $\fT^{**}$.
\end{lemma}
\begin{proof}
Put $\fL_\psi=\{t\in \fT:\psi(t^*t)=0\}$. The Schwarz inequality implies that this is a closed left ideal in $\fT$, and as such it admits an increasing right approximate identity $(e_i)$ consisting of positive contractions (see \cite[Lemma 2.1]{blecher2004} for a proof). Upon passing to a subnet if necessary, we may assume that $(e_i)$ converges to some $q\in \fT^{**}$ in the weak-$*$ topology. Clearly $0\leq q\leq I$. Furthermore, because
\[
\fL_\psi=\fT^{**}(I-\fl_\psi)\cap \fT
\]
we may infer that $q(I-\fl_\psi)=q$.  For each $j$, we have $e_j\in \fL_\psi$ so the net $(e_je_i)_i$ converges to $e_j$ in the norm topology of $\fT$ on one hand, and to $e_j q$ in the weak-$*$ topology of $\fT^{**}$ on the other. Therefore, $e_j q=e_j$ for every $j$. Taking weak-$*$ limits yields $q^2=q$. Thus, $q\in \fT^{**}$ is a self-adjoint projection with $q\leq I-\fl_\psi$.

We now claim that $q=I-\fl_\psi$. To see this, we fix a state $\phi$ on $\fT$ such that $\phi(q)=0$, and we aim to show that $\phi(I-\fl_\psi)=0$ as well.  Let $t\in \fL_\psi$. Then, we find
\begin{align*}
\phi(t^*t)=\lim_i\phi(e_i t^* te_i)\leq \|t\|^2\lim_i\phi( e_i^2)\leq \|t\|^2\lim_i\phi( e_i)=\|t\|^2\phi(q)=0.
\end{align*}
We conclude that $\fL_\psi\subset \fL_\phi$. Since $\psi$ is pure, we find
\[
\ker\psi=\fL_\psi+\fL_\psi^*\subset \fL_\phi+\fL_\phi^*\subset \ker \phi
\]
by \cite[Proposition 2.9.1]{dixmier1977}. 
Since both $\psi$ and $\phi$ are states, this forces $\psi=\phi$, so that $\phi(I-\fl_\psi)=0$. We conclude that $I-\fl_\psi\leq q$, so in fact $I-q=\fl_\psi$. By definition of $q$ we see that $I-q=\fl_\psi$ is indeed closed.
\end{proof}

We can now prove the main result of this section.

\begin{theorem}\label{T:peakdiff}
Let $\A\subset B(\H)$ be a  unital operator algebra and let $\omega$ be a state on $\rC^*(\A)$. Consider the following statements.
\begin{enumerate}[{\rm (i)}]
\item The left support projection $\fl_\omega$ is closed and lies in $\A^{\perp\perp}$. 
\item The left support projection $\fl_\omega$ is an $\A$-peak projection.
\item The state $\omega$ is an $\A$-peak state.
\item The GNS representation of $\omega$ is a local $\A$-peak representation.
\end{enumerate}
Then, we have that 
\[
{\rm (i)}\Leftarrow{\rm (ii)}\Rightarrow{\rm (iii)\Rightarrow }{\rm (iv)}.
\]
When $\A$ is separable, we have 
\[
{\rm (i)}\Leftrightarrow{\rm (ii)}\Rightarrow{\rm (iii)\Rightarrow }{\rm (iv)}.
\]
\end{theorem}
\begin{proof}

(ii) $\Rightarrow$ (i): This follows from \cite[Lemma 3.6 and Proposition 5.6]{hay2007}.

(i) $\Rightarrow$ (ii): When $\A$ is separable, this follows from Lemma \ref{L:purestateopen} along with Theorem \ref{T:ncGlick}.

(ii) $\Rightarrow$ (iii): There is a contraction $a\in \A$ with the property that $a\fl_\omega=\fl_\omega$ and if $\phi$ is a state on $\rC^*(\A)$ with $\phi(a^*a)=1$, then $\phi(\fl_\omega)=1$ as well. Lemma \ref{L:supppure} thus shows that $\omega$ is an $\A$-peak state.

(iii) $\Rightarrow$ (iv): Let $\pi$ denote the GNS representation of $\omega$. By assumption,  there is a contraction $a\in \A$ such that
\[
\omega(a^*a)=1>\phi(a^*a)
\]
for every pure state $\phi$ on $\rC^*(\A)$ with $\phi\neq \omega$. Let $\sigma:\rC^*(\A)\to B(\H_\sigma)$ be an irreducible $*$-representation  unitarily inequivalent to $\pi$. Then, for every unit vector $\eta\in \H_\sigma$, the state 
\[
t\mapsto \langle\sigma(t)\eta,\eta \rangle, \quad t\in \rC^*(\A)
\]
is pure and distinct from $\omega$, so that
\[
\langle \sigma(a^*a)\eta,\eta\rangle<1.
\]
Consequently, we see that $\|P_G \sigma(a^*a)|_G\|<1$ for every finite-dimensional subspace $G\subset \H_\sigma$. By virtue of the Schwarz inequality, we find $\|P_G \sigma(a)|_G\|<1$. 
\end{proof}

The implications in the previous theorem cannot be reversed in general, as we illustrate next.

\begin{example}\label{E:peakimplications}
Consider the $\rC^*$-algebra $\bM_2$, and denote by $E_{ij}$ the usual matrix units. Let $\A=\bC I+\bC E_{12}$, which is easily seen to be a unital subalgebra of $\bM_2$ with $\rC^*(\A)=\bM_2$. Let $\{e_1,e_2\}$ denote the standard orthonormal basis of $\bC^2$. Consider the states $\omega$ and $\psi$ on $\bM_2$ defined as
\[
\psi(t)=\langle te_1,e_1\rangle \qand \omega(t)=\langle te_2,e_2\rangle
\]
for every $t\in \bM_2$. 

Let $a=E_{12}\in \A$, which satisfies $a^*a=E_{22}$. Clearly, we have $\omega(a^*a)=1$. Next, let $\phi$ be a pure state on $\bM_2$ such that $\phi(a^*a)=1$. There exists a unit vector $\xi\in \bC^2$ such that
\[
\phi(t)=\langle t\xi,\xi\rangle,\quad t\in \bM_2.
\]
The condition $\phi(a^*a)=1$ implies that $E_{22}\xi=\xi$, so $\xi$ is a unimodular scalar multiple of $e_2$ and $\phi=\omega$. We thus conclude that $\omega$ is an $\A$-peak state.

Given $t\in \bM_2$, we see that $\omega(t^*t)=0$ if and only if 
$te_2=0.$ Therefore, we infer that
\begin{align*}
\{t\in \rC^*(\A):\omega(t^*t)=0\}&=\{t\in\bM_2 :  tE_{22}=0\}=\rC^*(\A)(I-E_{22}).
\end{align*}
This shows that $\mathfrak{l}_\omega = E_{22}$, which does not lie in $\A = \A^{\perp\perp}$. Hence, the implication (iii) $\Rightarrow$ (ii) from Theorem \ref{T:peakdiff} generally fails.

Next, observe that up to unitary equivalence, the identity representation is the only irreducible $*$-representation of $\bM_2$. In particular, it is trivially a local $\A$-peak representation. Nevertheless, $\psi$ is not an $\A$-peak state. Indeed, assume that $a\in \A$ satisfies $\|a\|=1$ and $\psi(a^*a)=1$. Write $a=\begin{bmatrix} \lambda & \mu \\ 0 & \lambda\end{bmatrix}$ for some $\lambda,\mu\in \bC$. We find
\[
|\lambda|=\|ae_1\|=\psi(a^*a)^{1/2}=1 
\]
so the condition $\|a\|=1$ forces $\mu=0$. In other words, $a$ is a unimodular scalar multiple of the identity, so in fact $\phi(a^*a)=1$ for every state $\phi$ on $\bM_2$.  This shows that the implication (iv) $\Rightarrow$ (iii) from Theorem \ref{T:peakdiff} generally fails.

\qed
\end{example}

As another consequence of Theorem \ref{T:peakdiff}, we show that for separable $\rC^*$-algebras, all pure states are peak states. 

\begin{corollary}\label{C:C*peak}
Let $\fT$ be a separable unital $\rC^*$-algebra. Then, any pure state on $\fT$ is a $\fT$-peak state, and any irreducible $*$-representation of $\fT$ is a local $\fT$-peak representation. 
\end{corollary}
\begin{proof}
The left support projection of any pure state on $\fT$ is closed by Lemma \ref{L:purestateopen}, and plainly lies in $\fT^{\perp\perp}=\fT^{**}$. Theorem \ref{T:peakdiff} then implies the desired statements. 
\end{proof}

Let $\A$ be separable, let $\omega$ be a pure state of $\rC^*(\A)$ and let $\pi$ denote its GNS representation, which is necessarily irreducible. Assume that $\pi$ is finite-dimensional. By virtue of Theorems \ref{T:locpeak} and \ref{T:peakdiff}, we see that $\pi$ is a local $\A$-peak representation when either $\fs_\pi$ or $\fl_\omega$ lies in $\A^{\perp\perp}$. In the remainder of this section, we clarify the relationship between these two conditions.  First, we show that it is possible for all irreducible $*$-representations of $\rC^*(\A)$ to have support projection lying in $\A^{\perp\perp}$, while some pure states have their left support projection outside of $\A^{\perp\perp}$.

\begin{example}\label{E:ell2}
Consider the unital operator algebra $\A\subset \bM_2$ from Example \ref{E:peakimplications}. All irreducible $*$-representations of $\rC^*(\A)$ are unitarily equivalent to the identity representation, whose support projection is simply the unit and so lies in $\A^{\perp\perp}$. On the other hand, the pure state $\psi$ defined in Example \ref{E:peakimplications} was shown not to be an $\A$-peak state, so by Theorem \ref{T:peakdiff} and Lemma \ref{L:purestateopen} it follows that $\fl_\psi\notin \A^{\perp\perp}$.
\qed
\end{example}

In the reverse direction, there are  pure states with left support projection in $\A^{\perp\perp}$ whose GNS representation is a local $\A$-peak representation with support projection outside of $\A^{\perp\perp}$. 

\begin{example}\label{E:ell1}
Consider the unital operator algebra $\A\subset \bM_2\oplus \bC$ consisting of all elements of the form
\[
\begin{bmatrix} \lambda & x \\ 0 & \mu\end{bmatrix}\oplus \lambda
\]
where $x,\lambda,\mu\in \bC$. It easily follows that $\rC^*(\A)=\bM_2\oplus \bC$. Let $\pi:\rC^*(\A)\to\bM_2$ and $\sigma:\rC^*(\A)\to \bC$ be the irreducible $*$-representations defined by
\[
\pi(a\oplus \lambda)=a \qand \sigma(a\oplus \lambda)=\lambda 
\]
for every $a\in \bM_2, \lambda\in \bC$.  We find $\fs_\pi=I_2\oplus 0$, which does not belong to $\A=\A^{\perp\perp}$. 

Up to unitary equivalence, the only irreducible $*$-representations of $\rC^*(\A)$ are $\pi$ and $\sigma$. It is readily seen that $\|\sigma^{(n)}(B)\|\leq \|\pi^{(n)}(B)\|$ for every $B\in \bM_n(\A)$ and every $n\geq 1$, so that $\sigma$ is not a local $\A$-peak representation. On the other hand, if we put $c=\begin{bmatrix} 0 & 1 \\ 0 & 0 \end{bmatrix}\oplus 0\in \A$, then we have
\[
\|\sigma(c)\|=0<\|\pi(c)\|=1
\]
so that $\pi$ is a local $\A$-peak representation. 

Next, let $\{e_1,e_2\}$ denote the standard orthonormal basis of $\bC^2$ and let $\omega$ be the pure state on $\rC^*(\A)$ defined by
\[
\omega(t)=\langle \pi(t)e_2,e_2 \rangle, \quad t\in \rC^*(\A).
\]
Arguing as in Example \ref{E:peakimplications}, it is readily seen that $\ell_\omega=E_{22}\oplus 0\in \A$.
\qed
\end{example}
Another consequence of the previous example is that, even in the finite-dimensional setting, there are pure states whose left support projection is an $\A$-peak projection, yet the support projection of their GNS representation is not an $\A$-peak projection. In other words, the implication property (iv) in Theorem \ref{T:peakdiff} cannot be strengthened in this fashion.

In contrast with the previous example, we will show that if \emph{all} pure states on $\rC^*(\A)$ have their left support projection in $\A^{\perp\perp}$, then we can guarantee that $\fs_\pi\in \A^{\perp\perp}$ for all irreducible finite-dimensional $*$-representations $\pi$ as well. The following standard fact is required. Given a collection $\F$ of projections in a von Neumann algebra, we denote by $\bigvee \{p:p\in \F\}$ their supremum.

\begin{lemma}\label{L:suppstates}
Let $\fT$ be a unital $\rC^*$-algebra and let $\pi:\fT\to B(\H_\pi)$ be an irreducible finite-dimensional  $*$-representation of $\fT$. Then, there is a finite collection $\F$ of pure states on $\fT$ such that $\fs_\pi=\bigvee \{\fl_\omega:\omega\in \F\}$.
\end{lemma}
\begin{proof}
Let $\{e_1,\ldots,e_n\}$ be an orthonormal basis for $\H_\pi$. Since $\pi$ is irreducible, the state $\omega_j$ on $\fT$ defined as
\[
\omega_j(t)=\langle \pi(t) e_{j},e_{j}\rangle, \quad t\in \fT
\]
is pure. Put $p=\bigvee_{j=1}^n \fl_{\omega_j}\in \fT^{**}$. Because $\H$ is finite-dimensional, we see that $\pi^{**}:\fT^{**}\to B(\H)$ is the unique weak-$*$ continuous extension of $\pi$ to $\fT^{**}$.
Then, 
\[
\langle \pi^{**}(p)e_j,e_j\rangle=\omega_j(p)\geq \omega_j(\fl_{\omega_j})=1
\]
for every $1\leq j\leq n$, whence $\pi^{**}(p)=I$. Thus, $\fs_\pi\leq p$. Conversely, for $1\leq j\leq n$ we see that
\[
\omega_j(\fs_\pi)=\langle \pi^{**}(\fs_\pi)e_j,e_j\rangle=1
\]
so that $\fl_{\omega_j}\leq \fs_\pi$. Therefore, $p\leq \fs_\pi$.
\end{proof}

We can now give the desired result.

\begin{theorem}\label{T:suppA}
Let $\A\subset B(\H)$ be a separable unital operator algebra. Assume that the left support projection of every pure state on $\rC^*(\A)$ is an $\A$-peak projection. Then, the support projection of every finite-dimensional irreducible $*$-representation $\pi$ of $\rC^*(\A)$ is an $\A$-peak projection.
\end{theorem}
\begin{proof}
By Theorem \ref{T:peakdiff}, we see that the left support projection of every pure state on $\rC^*(\A)$ lies in $\A^{\perp\perp}$. 
Let $\pi$ be a finite-dimensional irreducible $*$-representation of $\rC^*(\A)$.  By Lemma \ref{L:suppstates}, we conclude that $\fs_\pi\in\A^{\perp\perp}$, so that $\fs_\pi$ is an $\A$-peak projection by Theorem \ref{T:ncGlick}.
\end{proof}

We note here that the previous result does not follow directly from Theorem \ref{T:peakdiff} (see the discussion following Example \ref{E:ell1}). Moreover, the converse does not hold, as demonstrated by Example \ref{E:ell2}.

\section{Peak points in the Choquet boundary}\label{S:peakChoquet}

In the classical case of uniform algebras, peak points necessarily lie in the Choquet boundary. In this section, we investigate non-commutative analogues of this fact for the different types of peak points studied in Section \ref{S:ncpeak}.  Let us next give some details about what this means.

Let $\A\subset B(\H)$ be a unital operator algebra. A unital $*$-representation $\beta:\rC^*(\A)\to B(\H_\beta)$ is said to have the \emph{unique extension property with respect to $\A$} if $\beta$ is the unique completely contractive extension to $\rC^*(\A)$ of $\beta|_\A$. If in addition $\beta$ is irreducible, then we say that $\beta$ is a \emph{boundary representation for $\A$}. 
 
Although our focus in this paper is on operator algebras, in what follows we will also deal with operator systems: unital self-adjoint subspaces of unital $\rC^*$-algebras. The notion of boundary representations also makes sense in this context. Indeed, let $\S\subset B(\H)$ be an operator system. An irreducible $*$-representation $\beta:\rC^*(\S)\to B(\H_\beta)$ is said to be a boundary representation for $\S$ if $\beta$ is the unique completely positive extension to $\rC^*(\S)$ of $\beta|_\S$. 

These two definitions of boundary representations are in fact very much compatible. Indeed, a unital completely contractive linear map $\phi:\A\to B(\H_\phi)$ admits a \emph{unique} unital completely positive extension  $\tilde\phi:\A+\A^*\to B(\H_\phi)$ \cite[Proposition 1.2.8]{arveson1969}. This implies that an irreducible $*$-representation $\beta:\rC^*(\A)\to B(\H_\beta)$ is a boundary representation for $\A$ if and only if it is a boundary representation for the operator system $\A+\A^*$.

As discussed in Section \ref{S:intro}, the usual non-commutative analogue of the Choquet boundary consists of the collection of all boundary representations for $\A$. We are thus led to ask the following questions.

\begin{enumerate}[{\rm (a)}]
\item Let $\pi$ be a local $\A$-peak representation of $\rC^*(\A)$. Must $\pi$ be a boundary representation for $\A$?
\item Let $\omega$ be an $\A$-peak state on $\rC^*(\A)$. Must $\omega|_\A$ admit a unique state extension to $\rC^*(\A)$?
\item Let $q\in \rC^*(\A)^{**}$ be a central $\A$-peak projection. Must $q$ coincide with the support projection of some unital $*$-representation of $\rC^*(\A)$ with the unique extension property with respect to $\A$?
\end{enumerate}

Question (a) was already resolved in previous work; we record the statement here for future reference.

\begin{theorem}\label{T:locpeakbdry}
Let $\A\subset B(\H)$ be a unital  operator algebra and let $\pi$ be an irreducible $*$-representation of $\rC^*(\A)$. If $\pi$ is a local $\A$-peak representation, then it is a boundary representation for $\A$. 
\end{theorem}
\begin{proof}
This is \cite[Theorem 3.1]{CTh2020fdim}.
\end{proof}

We now move on to states. In this case again, the question is easily answered, albeit in the negative this time.

\begin{example}\label{E:peakstateuep}
Consider the algebra $\A\subset \bM_2$ along with the two states $\omega$ and $\psi$ from Example \ref{E:peakimplications}. Recall that $\omega$ was shown therein to be an $\A$-peak state. Observe now that the states $\omega$ and $\psi$ clearly agree on $\A$. This shows that the property of $\omega$ being an $\A$-peak state does not imply that it has a unique extension to a state on $\bM_2$.
\qed
\end{example}

By changing Question (b) slightly, a positive result can be established.

\begin{proposition}\label{P:statepeakbdry}
Let $\A\subset B(\H)$ be a unital  operator algebra and let $\omega$ be a pure state on $\rC^*(\A)$.  If $\fl_\omega\in \A^{\perp\perp}$, then $\omega|_\A$ has a unique state extension to $\rC^*(\A)$. In particular, this holds if $\fl_\omega$ is an $\A$-peak projection.
\end{proposition}
\begin{proof}
Let $\psi$ be a state on $\rC^*(\A)$ such that $\psi|_\A=\omega|_\A$. It follows that $\psi(\xi)=\omega(\xi)$ for every $\xi\in \A^{\perp\perp}$, so by assumption we find $\psi(\fl_\omega)=\omega(\fl_\omega)=1$. Applying Lemma \ref{L:supppure}, we find $\psi=\omega$. The second statement follows from Theorem \ref{T:peakdiff}.
\end{proof}

In the rest of this section, we tackle Question (c), which turns out to be quite a bit more elusive than the previous two. We start with a fact that will be used repeatedly.
 
 \begin{lemma}\label{L:peaktrick}
 Let $\fT$ be a unital $\rC^*$-algebra and let $\A\subset \fT$ be a unital subalgebra. Let $\pi$ be a unital $*$-representation of $\fT$ and let $a\in \A$ be a contraction peaking at $\fs_\pi$.  Let $\sigma:\fT\to B(\H_\sigma)$ be an irreducible $*$-representation with the property that there is a finite-dimensional subspace $F\subset \H_\sigma$ such that
 \[
 \|P_F\sigma(a^*a)|_F\|\geq \|\pi(a^*a)\|.
 \]
  Then,  we have $\ker \pi\subset \ker \sigma$.
 \end{lemma}
 \begin{proof}
Since $a\fs_\pi=\fs_\pi$ and $\pi^{**}(\fs_\pi)=I$, we see that $\pi(a)=I$. By assumption, we find
  \[
 \|P_F\sigma(a^*a)|_F\|\geq \|\pi(a)\|^2=1
 \]
and so by choice of $a$ and the definition of an $\A$-peak projection, we must have
 \[
\| P_F \widehat\sigma(\fs_\pi)|_F\|=1.
\]
In particular, $\widehat\sigma(\fs_\pi)$ is a non-zero projection commuting with $\sigma(\fT)$. Since $\sigma$ is irreducible, this forces $\widehat\sigma(I-\fs_\pi)=0$, or 
$\ker \pi^{**}\subset \ker \widehat \sigma$. In particular, we infer that $\ker \pi\subset \ker \sigma$.  
 \end{proof}

We also need the following standard fact.

\begin{lemma}\label{L:kerA}
Let $\A\subset B(\H)$ be a unital operator algebra and let $\pi:\rC^*(\A)\to B(\H_\pi)$ be a unital $*$-representation. Assume that $\fs_\pi\in \A^{\perp\perp}$. Then, 
$
\rC^*(\A\cap \ker\pi)=\ker \pi.
$
\end{lemma}
\begin{proof}
Throughout the proof, we let $\fT=\rC^*(\A)$ and $p=I-\fs_\pi$. Since $p$ is central in $\fT^{**}$, the map $\lambda:\fT\to\fT^{**}$ defined as
\[
\lambda(t)=t p,\quad t\in \fT
\]
is a $*$-homomorphism, so that
\[
\fT p=\lambda(\fT)=\rC^*(\lambda(\A))=\rC^*(\A p).
\]
On the other hand,
\[
\ker \pi=\fT\cap \ker (\pi^{**})=\{t\in \fT:t p=t\}\subset \fT p
\]
whence 
\begin{equation}\label{Eq:Ap1}
\ker \pi\subset \rC^*(\A p).
\end{equation}
Since $\fs_\pi=I-p$ is closed, we know from \cite[Theorem 3.3]{blecher2013} that
\[
\A^{\perp\perp}p=(\A\cap \ker \pi)^{\perp\perp}
\]
so in particular 
\begin{equation}\label{Eq:Ap2}
\A p\subset \A^{\perp\perp}p=(\A\cap \ker \pi)^{\perp\perp}\subset (\rC^*(\A\cap \ker \pi))^{\perp\perp}.
\end{equation}
Combining \eqref{Eq:Ap1} and \eqref{Eq:Ap2}, we infer that 
\[
\ker \pi\subset  \rC^*(\A p)\subset  (\rC^*(\A\cap \ker \pi))^{\perp\perp}
\]
 so that $\ker \pi$ lies in the closure of $\rC^*(\A\cap \ker \pi)$ in the weak topology of $\fT$. Since $\rC^*(\A\cap \ker \pi)$ is a norm-closed subspace, it is also weakly closed, and we conclude that $\ker\pi\subset \rC^*(\A\cap \ker \pi)$. The reverse inclusion is trivial.
\end{proof}

Before proceeding, we need to introduce some terminology from matrix convexity theory (see \cite{farenick2000} or \cite{HL2021} and the references therein). Let $\S$ be an operator system. For each positive integer $n\geq 1$, let $K_n(\S)$ denote the set of all unital completely positive linear maps $\psi:\S\to\bM_n$. Then, each $K_n$ is a convex set that is closed in the topology of pointwise norm convergence. In fact, more is true: the sequence $K(\S)=(K_n(\S))$ forms a compact  \emph{matrix convex set} on $\S^*$ (endowed with the weak-$*$ topology). By \cite[Theorem A]{farenick2000}, the \emph{matrix extreme points} of $K(\S)$ are the pure completely positive maps as introduced in \cite{arveson1969}.

\begin{lemma}\label{L:matconv}
Let $\S\subset B(\H)$ be an operator system, let $\X\subset \S$ be a finite-dimensional operator subsystem, let $\F$ be a finite-dimensional Hilbert space and let $\phi:\X\to B(\F)$ be a unital completely positive map. Let $a\in \X$ be a contraction such that $\phi(a)=I$. Then, there is a finite set $\sB$ of boundary representations for $\S$ on $\rC^*(\S)$ such that 
\[
\X \cap \left(\bigcap_{\beta\in \sB}\ker \beta\right)\subset \ker \phi
\]
and for every $\beta\in \sB$ there is a finite-dimensional subspace $G_\beta$ such that 
\[
\|P_{G_\beta}\beta(a)|_{G_\beta}\|=1.
\]
\end{lemma}
\begin{proof}
Apply  \cite[Theorem 2.9]{HL2021} to the matrix convex set $K(\X)$ to find finitely many pure unital completely positive maps $\psi_1,\ldots,\psi_{r}$ in $K(\X)$ along with non-zero operators $V_1,\ldots,V_{r}$ such that $\sum_{i=1}^r V^*_i V_i=I$ and
\[
\phi(b)=\sum_{i=1}^{r}V^*_i\psi_{i}(b)V_i, \quad b\in \X.
\]
Now, by \cite[Theorem B]{farenick2000}, for each $1\leq i\leq r$ there is a pure unital completely positive map $\theta_i$ in $K(\S)$ extending $\psi_{i}$ (alternatively, see \cite[Corollary 2.3]{kleski2014bdry}). Hence,
\[
\phi(b)=\sum_{i=1}^{r}V^*_i\theta_{i}(b)V_i, \quad b\in \X.
\]
Fix $1\leq j\leq r$ and let $\xi\in \F$ be a unit vector such that $V_j\xi\neq 0$. 
Using that $\phi(a)=I$ and the Cauchy--Schwarz inequality, we find
 \begin{align*}
 1&=\langle \phi(a)\xi,\xi \rangle= \sum_{i=1}^r \langle \theta_i(a)V_i \xi,V_i \xi \rangle\leq \sum_{i=1}^r \|\theta_i(a) V_i\xi\| \|V_i\xi\|\\
 &\leq \sum_{i=1}^r \|V_i\xi\| ^2=\|\xi\|^2=1.
  \end{align*}
  We conclude that $\theta_i(a)V_i\xi=V_i\xi$ for every $1\leq i \leq r$. Since $V_j\xi\neq 0$, this implies that $\|\theta_j(a)\|=1$.
Furthermore, invoking \cite[Theorem 2.4]{DK2015}, we see that each $\theta_i$ dilates to some boundary representation $\beta_i$ for $\S$ on $\rC^*(\S)$. Since the range of $\theta_i$ is finite-dimensional, this implies the existence of a finite-dimensional subspace $G_i$ such that $\|P_{G_i}\beta_i(a)|_{G_i}\|= \|\theta_i(a)\|=1$.
Finally, we see that
 \[
\X\cap \left(\bigcap_{i=1}^r\ker \beta_i\right)\subset \X\cap\left( \bigcap_{i=1}^r\ker \theta_i \right)\subset \ker \phi.
\]
\end{proof}

We now arrive at our next main result, showing that $\A$-peak projections are related in some way to the Choquet boundary of $\A$, and thereby addressing Question (c).

\begin{theorem}\label{T:peakprojker}
Let $\A\subset B(\H)$ be a unital operator algebra. Let $\pi:\rC^*(\A)\to B(\H_\pi)$ be a unital $*$-representation  whose support projection is an $\A$-peak projection. Then, there is a set $\sB_\pi$ of boundary representations for $\A$ on $\rC^*(\A)$ such that if we set $\fJ_\pi=\bigcap_{\beta\in \sB_\pi}\ker \beta$, then 
\[
\A\cap \fJ_\pi\subset \ker \pi\subset \fJ_\pi \qand \rC^*(\A\cap \fJ_\pi)=\ker \pi.
\]
\end{theorem}
\begin{proof}
Throughout the proof, we let $\S=\A+\A^*\subset B(\H)$.
Choose a contraction $a\in \A$ peaking at $\fs_\pi$. Let $\O$ denote the set of finite-dimensional, unital, self-adjoint subspaces of $\S$ containing $a$, partially ordered by inclusion. Let $\E$ denote the set of finite-dimensional subspaces of $\H_\pi$, partially ordered by inclusion. Consider the directed set $\Lambda=\O\times \E$, equipped with the usual product partial order. Given $\lambda=(\X,\F)\in \Lambda$, define a unital completely positive map $\phi_\lambda:\X\to B(\F)$ as
\[
\phi_{\lambda}(b)=P_{\F}\pi(b)|_\F, \quad b\in \X.
\]
Note that $\phi_\lambda(a)=I$ since $\pi(a)=I$.
Apply Lemma \ref{L:matconv} to find a finite subset $\sB_\lambda$ of boundary representations for $\A$ on $\rC^*(\A)$ such that 
\begin{equation}\label{Eq:kerincl}
\X \cap \left(\bigcap_{\beta\in \sB_\lambda}\ker \beta\right)\subset \ker \phi_\lambda
\end{equation}
and for every $\beta\in \sB_\lambda$ there is a finite-dimensional subspace $G_{\beta}$ such that 
\[
\|P_{G_\beta}\beta(a)|_{G_\beta}\|=1.
\]
By the Schwarz inequality, we find 
\[
\|P_{G_\beta}\beta(a^*a)|_{G_\beta}\|=1, \quad \beta\in \sB_\lambda
\]
so Lemma \ref{L:peaktrick} yields
$
\ker \pi\subset \bigcap_{\beta\in \sB_\lambda}\ker \beta.
$
If we put $\sB_\pi=\bigcup_{\lambda\in \Lambda}\sB_\lambda$ and $\fJ_\pi=\bigcap_{\beta\in \sB_\pi}\ker \beta$, then $\ker\pi\subset \fJ_\pi$.
Let $b\in  \A\cap \fJ_\pi$. 
 Choose $\lambda_0\in \Lambda$ with the property that $b\in \X$ for every $\lambda=(\X,\F)\geq \lambda_0$. By Equation \eqref{Eq:kerincl}, we see that $b\in \bigcap_{\lambda\in \Lambda,\lambda\geq \lambda_0}\ker \phi_\lambda$.
On the other hand, we find
\[
\pi(b)=\lim_{\lambda\in \Lambda,\lambda\geq \lambda_0}\phi_\lambda(b)
\]
in the strong operator topology, so that $\pi(b)=0$. This shows that $\A\cap \fJ_\pi\subset \ker \pi$, so we conclude that
\[
\A\cap \fJ_\pi\subset \ker \pi\subset \fJ_\pi.
\]
Finally, using the previous inclusions we infer $\A\cap \ker \pi=\A\cap \fJ_\pi$. By \cite[Lemma 3.6]{hay2007}, we see that $\fs_\pi\in \A^{\perp\perp}$,  whence $\ker \pi=\rC^*(\A\cap \fJ_\pi)$ by Lemma \ref{L:kerA}.
\end{proof}

We close this section by showing that the set $\sB_\pi$ in the previous theorem cannot generally be taken be a singleton (or even  finite), and that the conclusion cannot be improved to $\fJ_\pi=\ker \pi$.

\begin{example}\label{E:singleton}
Let $\fT_1$ be the Toeplitz algebra and $\A_1\subset \fT_1$ be the unital subalgebra defined in Example \ref{E:peakAD}. Let $\pi$ denote the identity representation of $\fT_1$, which is irreducible since $\fT_1$ contains the compact operators. There is a surjective $*$-homomorphism $q:\fT_1\to \rC(\bT)$ that is completely isometric on $\A_1$. It follows from this that the boundary representations for $\A_1$ are precisely those of the form $\chi_\zeta\circ q$, where for $\zeta\in \bT$, the map $\chi_\zeta:\rC(\bT)\to\bC$ is the character of evaluation at $\zeta$.

Since $\pi$ is injective, it is trivial that its support projection lies in $\A_1^{\perp\perp}$, and thus is an $\A_1$-peak projection by Theorem \ref{T:ncGlick}. By Theorem \ref{T:peakprojker}, there is a set $\sB\subset \bT$ such that if we set $\fJ=\bigcap_{\zeta\in \sB}\ker(\chi_\zeta\circ q)$, then
\[
\rC^*(\A_1\cap \fJ)=\ker \pi=\{0\}.
\]
Indeed, this can be achieved, for instance, by choosing $\sB=\bT$.

If $\sB\subset \bT$ were finite, then  $\prod_{\zeta\in \sB}(S-\zeta I)$ would be a non-zero element of $\A_1\cap \fJ$, contradicting the previous equality. Thus, $\sB$ must be infinite. Furthermore, we see that the compact operators are contained in $\ker (\chi_\zeta\circ q)$ for every $\zeta\in \bT$, and in particular in $\fJ$. This shows that $\ker\pi\neq \fJ$. 
\qed
\end{example}

We remark here that given an irreducible $*$-representation $\pi$ whose support projection is an $\A$-peak projection, the fact that $\ker \pi$ coincides with the kernel of a boundary representation for $\A$ does not guarantee that $\pi$ itself is a boundary representation; see for instance \cite[Continuation of Example 2.7]{muhly1998tensor} or \cite[Example 6.6.3]{DK2019}.

\section{The Bishop property}\label{S:Bishop}

In Section \ref{S:intro}, we discussed how, for uniform algebras, a theorem of Bishop shows that every point in the Choquet boundary is necessarily a peak point. Correspondingly, given a unital $\rC^*$-algebra and a unital operator algebra $\A\subset \fT$, we say that $\A$ has the \emph{Bishop property} in $\fT$ whenever all boundary representations for $\A$ on $\fT$ are local $\A$-peak representations. Due to an earlier argument, this property is seen to always hold for separable $\rC^*$-algebras.

\begin{corollary}\label{C:C*Bishop}
Any separable unital $\rC^*$-algebra has the Bishop property in itself.
\end{corollary}
\begin{proof}
This follows at once from Corollary \ref{C:C*peak}.
\end{proof}

Our next order of business is to illustrate that the Bishop property is not automatic for non self-adjoint operator algebras. Consequently, Bishop's theorem does not extend to the non-commutative context. The following is inspired by \cite[page 43]{phelps2001}.
 
 \begin{example}\label{E:nonBishop}
 Let $E$ be a locally convex topological vector space and let $K\subset E$ be a compact convex subset containing an extreme point $\xi$ which is  \emph{not exposed}. That is, there is no continuous linear functional $\phi:E\to\bC$ such that 
 \[
 \re \phi(x)<\re \phi(\xi)
 \]
 for every $x\in K$ with $x\neq \xi$.
 
 Let $E^*$ denote the space of continuous linear functionals on $E$. Let $\S \subset \rC(K)$ be the norm-closed operator system generated by the  unit and $E^*$. 
 Consider the unital operator algebra $\A\subset \bM_2(\rC(K))$ consisting of elements of the form
 \[
 \begin{bmatrix}
 \lambda & f\\ 0 & \mu
 \end{bmatrix}
 \]
 where $\lambda,\mu\in \bC$ and $f\in\S$. It is readily verified that $\rC^*(\A)=\bM_2(\rC(K))$. 
 
 Let $\beta:\rC(K)\to \bC$ denote the character of evaluation at $\xi$. Since $\xi$ is an extreme point of $K$,  $\beta$ is  is a boundary representation for $\S$ \cite[Lemmas 2.2 and 4.1]{bishop1959}. It then follows from \cite{hopenwasser1973} that $\beta^{(2)}$ is a boundary representation for $\A$. 
 
 Assume that $\beta^{(2)}$ is a local $\A$-peak representation. Note that all irreducible $*$-representations of $\bM_2(\rC(K))$ are unitarily equivalent to one of the form $\pi^{(2)}$, where $\pi:\rC(K)\to\bC$ is evaluation at some $x\in K$. Thus,
 there is an integer $n\geq 1$ and an element $B\in \bM_n(\A)$ such that 
  \[
  \|B(\xi)\|=1>\|B(x)\|, \quad x\neq \xi.
 \]
 Note that $B\in \bM_{2n}(\S)$, so there is $\Lambda\in\bM_{2n}$ and $F\in \bM_{2n}(E^*)$ such that $B=\Lambda+F$.
 Choose unit vectors $v,w\in \bC^{2n}$ with the property that $\|B(\xi)\|=\langle B(\xi)v,w\rangle$. We thus find
 \begin{align*}
 \re \langle \Lambda v,w \rangle+\re \langle F(\xi) v,w \rangle&=\re \langle B(\xi)v,w\rangle=\|B(\xi)\|\\
 &> \|B(x)\|\geq \re \langle B(x)v,w\rangle\\
 &= \re \langle \Lambda v,w \rangle+\re \langle F(x) v,w \rangle
 \end{align*}
 or
 \[
 \re \langle F(\xi) v,w \rangle>\re \langle F(x) v,w \rangle
 \]
 for every $x\in K, x\neq \xi$. This contradicts the fact that $\xi$ is not an exposed point of $K$. Consequently, $\beta^{(2)}$ is a boundary representation for $\A$ which is not a local $\A$-peak representation. Hence, $\A$ does not have the Bishop property inside of $\bM_2(\rC(K))$. 
 \qed
  \end{example}

We will show in Subsection \ref{SS:min} that the Bishop property encodes a certain minimality feature of the Choquet boundary for an operator algebra. 
For now, our aim is to identify non-commutative and non self-adjoint operator algebras with the Bishop property. We will be particularly interested in those algebras $\A$ for which $\rC^*(\A)$ contains the compact operators. Our analysis in this direction hinges on the following.

\begin{lemma}\label{L:Kbdrypeak}
Let $\A\subset B(\H)$ be a unital operator algebra with the property that $\rC^*(\A)$ contains the ideal of compact operators on $\H$. 
Then, the identity representation of $\rC^*(\A)$ is a boundary representation for $\A$ if and only if it is a local $\A$-peak representation. 
\end{lemma}
\begin{proof}
Let $\pi:\rC^*(\A)\to B(\H)$ be the identity representation and let $q:\rC^*(\A)\to \rC^*(\A)/\fK$ denote the quotient map by the compact operators. Assume that $\pi$ is a boundary representation for $\A$. We may apply \cite[Theorem 2.1.1]{arveson1972} to find an integer $n\geq 1$ and an element $B\in \bM_n(\A)$ such that $\|q^{(n)}(B)\|<\|B\|$. Let $\sigma$ be another irreducible $*$-representation of $\rC^*(\A)$ which is not unitarily equivalent to $\pi$. Then, $\sigma$ must annihilate the compact operators, so we may find an irreducible $*$-representation $\sigma'$ of $\rC^*(\A)/\fK$ such that $\sigma=\sigma'\circ q$.  Hence
\[
\|\sigma^{(n)}(B)\|\leq \|q^{(n)}(B)\|<\|B\|=\|\pi^{(n)}(B)\|.
\]
This shows that $\pi$ is a local $\A$-peak representation. The converse follows at once from Theorem \ref{T:locpeakbdry}.
\end{proof}

The reader will notice that the previous proof shows that if the identity representation is a boundary representation for $\A$, then it is in fact an $\A$-peak representation in the sense of \cite{arveson2011}. Since we do not make use of this stronger notion in our work, we do not pursue this here.

Next, we delve into a concrete class of operator algebras.

 Let $d\geq 1$ be a positive integer. Let $\fF$ be a Hilbert space of holomorphic functions on the open unit ball $\bB_d\subset \bC^d$. We require $\fF$ to be a reproducing kernel Hilbert space, meaning that there is a positive semi-definite function $k:\bB_d\times \bB_d\to\bC$ with the property that $k(\cdot, \lambda)\in \fF$ and
 \[
 f(\lambda)=\langle f,k(\cdot,\lambda)\rangle, \quad f\in \fF
 \]
  for every every $\lambda\in \bB_d$. We will assume that $\fF$ is a so-called \emph{regular unitarily invariant} space: the kernel $k$ has the form $k(z,\lambda)=\sum_{n=0}^\infty a_n \langle z,\lambda\rangle^n$ for some sequence $(a_n)$ of strictly positive numbers with $a_0=1$ and $\lim_{n\to\infty} a_n/a_{n+1}=1$. We assume in addition that the kernel $k$ satisfies the complete Pick property \cite{agler2002}.

 The resulting class of regular, unitarily invariant, complete Pick spaces is highly structured and contains many frequently studied examples arising naturally in function theory and operator theory, such as the Hardy space on the disc, the classical Dirichlet space, as well as the Drury--Arveson space. The reader may consult \cite{CH2018},\cite{CT2021},\cite{DH2020} for further detail.
 
A function $\theta:\bB_d\to\bC$ is a \emph{multiplier} for $\fF$ if $\theta f\in \fF$ for every $f\in \fF$. The corresponding multiplication operator $M_\theta:\fF\to\fF$ is then bounded.
In our setting, polynomials are always multipliers, and we define $\rA(\fF)$ to be the unital norm-closed subalgebra of $B(\fF)$ generated by the polynomial multipliers. Let also $\fT(\fF)=\rC^*(\rA(\fF))$.  It follows from \cite[Theorem 4.6]{GHX04} that $\fT(\fF)$ contains the ideal of compact operators $\fK$, and there is a $*$-isomorphism $\fT(\fF)/\fK\cong \rC(\bS_d)$ such that $M_p+\fK\mapsto p$ for any polynomial $p$ (here and below, $\bS_d$ denotes the unit sphere). Let $q:\fT(\fF)\to \fT(\fF)/\fK$ denote the quotient map. For each $\zeta\in \bS_d$, we let $\chi_\zeta: \rC(\bS_d)\to \bC$ denote the character of evaluation at $\zeta$.

We now record a well-known fact describing the representation theory of $\fT(\fF)$. 
We say that the kernel $k$ is \emph{unbounded} when the series $\sum_{n=0}^\infty a_n$ diverges. 
\begin{lemma}\label{L:repTF}
Let $\fF$ be a regular, unitarily invariant, complete Pick space on the unit ball. Then, the following statements hold.
\begin{enumerate}[{\rm (i)}]
\item An irreducible $*$-representation of $\fT(\fF)$ is unitarily equivalent to either the identity representation or to $\chi_\zeta\circ q$ for some $\zeta\in \bS_d$.

\item Assume that $\fF$ is not the Hardy space on the unit disc and that the kernel $k$ is unbounded. Then, the boundary representations for $\rA(\fF)$ are precisely the identity representation of $\fT(\fF)$ along with  $\chi_\zeta\circ q$ for each $\zeta\in \bS_d$.

\item If the kernel $k$ is bounded, then the identity representation of $\fT(\fF)$ is the only boundary representation for $\rA(\fF)$.
\end{enumerate}
 \end{lemma}
 \begin{proof}
 Statement (i) follows from basic representation theory of $\rC^*$-algebras (see for instance \cite[Lemma 3.3]{CH2018}) along with the fact that $\fT(\fF)/\fK\cong \rC(\bS_d)$.
 
 Statements (ii) and (iii) follow from  \cite[Theorem 6.2]{CH2018}.
 \end{proof}

We now arrive at another one of our main results.

\begin{theorem}\label{T:AHBishop}
Let $\fF$ be a regular, unitarily invariant, complete Pick space on the unit ball.  Then, $\rA(\fF)$ has the Bishop property inside of $\fT(\fF)$.\end{theorem}
\begin{proof}
If the kernel of $\fF$ is bounded, then by Lemmas \ref{L:Kbdrypeak} and \ref{L:repTF} we see that $\rA(\fF)$ indeed has the Bishop property inside of $\fT(\fF)$. 
Thus, for the remainder of the proof, we assume that the kernel of $\fF$ is unbounded.

If $\fF$ is the Hardy space on the unit disc, then it is well known that $\rA(\fF)$ is identified with  the disc algebra, and the boundary representations are the characters $\chi_\zeta\circ q$ for each $\zeta\in \bS_d$. 

If $\fF$ is not the Hardy space on the unit disc, then by Lemma \ref{L:repTF}, the boundary representations for $\rA(\fF)$ are precisely the identity representation, which is a local $\rA(\fF)$-peak representation by Lemma \ref{L:Kbdrypeak}, and the characters $\chi_\zeta\circ q$ for each $\zeta\in \bS_d$. 

Thus, in either case, it suffices to fix $\zeta\in \bS_d$ and to show that  $\pi=\chi_\zeta\circ q$ is a local $\A(\fF$)-peak representation.  

To see this, apply \cite[Theorem 8.1 and Proposition 9.2]{DH2020} to find a function $a:\ol{\bB_d}\to\bC$ such that $M_a$ is a contraction in $\rA(\fF)$ with the property that $a(\zeta)=1$ and $|a(w)|<1$ for every $w\in \bS_d, w\neq \zeta$. Based on an argument used in \cite[Proposition 6.4]{CD2016duality}, we define $b=\sum_{n=1}^\infty 2^{-n} a^n$. Clearly, $M_b$ is a contractive element of $\rA(\fF)$, and we still have $b(\zeta)=1$ and $|b(w)|<1$ for every $w\in \bS_d, w\neq \zeta$. Since bounded pointwise convergence to $0$ on $\bB_d$ coincides with weak-$*$ convergence to $0$ in $B(\fF)$ (see for instance \cite[Section 2.1]{DH2020}), it follows from the maximum modulus principle that the sequence $(M_a^n)$ converges to $0$ in the weak-$*$ topology of $B(\fF)$. In particular, if $F\subset \fF$ is a finite-dimensional subspace, then $(P_F M_a^n |_F)$ converges to $0$ in norm, whence $\|P_F M_b|_F\|<1$.  By Lemma \ref{L:repTF}, we conclude that $\pi$ is a local $\rA(\fF)$-peak representation.
\end{proof}

A few remarks are in order. First, \cite[Example 2]{CTh2020fdim} illustrates that the peaking behaviour in the previous statement cannot be strengthened to ``global'' peaking. Second, we mention that the proof above does not require the full strength of the complete Pick property of the kernel $k$. Indeed, the only instance where this property is used is to invoke \cite[Proposition 9.2]{DH2020}, which in fact only requires the so-called ``$2$-point Pick property".

At present, we have two classes of non self-adjoint operator algebras with the Bishop property: uniform algebras, and the algebras of multipliers considered in Theorem \ref{T:AHBishop}. Our next result is a mechanism to construct additional examples, by means of tensor products. Given two operator algebras, we denote by $\A\otimes \B$ their minimal tensor product; see \cite[Paragraph 2.2.2]{BLM2004}. Furthermore, recall that a unital $\rC^*$-algebra is said to be \emph{liminal} (or \emph{CCR}) if all of its irreducible $*$-representations are finite-dimensional.

\begin{theorem}\label{T:tensorBishop}
For $j=1,2$, let $\A_j\subset B(\H_j)$ be a unital operator algebra with the Bishop property inside of $\rC^*(\A_j)$. Assume  that $\rC^*(\A_1)$ is liminal. Then,  $\A_1\otimes \A_2$  has the Bishop property inside of $\rC^*(\A_1)\otimes \rC^*(\A_2)$.
\end{theorem}
\begin{proof}
First, since $\rC^*(\A_1)$ is liminal,  all irreducible $*$-representations of $\rC^*(\A_1)\otimes \rC^*(\A_2)$ are unitarily equivalent to one of the form $\pi_1\otimes \pi_2$, where $\pi_1$ and $\pi_2$ are irreducible $*$-representations of $\rC^*(\A_1)$ and $\rC^*(\A_2)$ respectively \cite[Lemma 10]{wulfsohn1963}.
Furthermore, it follows from  \cite{hopenwasser1978} that $\pi_1\otimes \pi_2$ is a boundary representation for $\A_1\otimes \A_2$ if and only if $\pi_1$ and $\pi_2$ are boundary representations for $\A_1$ and $\A_2$ respectively. Thus, we fix $\pi_1$ and $\pi_2$ boundary representations for $\A_1$ and $\A_2$ respectively, and we aim to show that $\pi_1\otimes \pi_2$ is a local $(\A_1\otimes \A_2)$-peak representation. 

For each $j=1,2$, the algebra $\A_j$ has the Bishop property inside of $\rC^*(\A_j)$, so there is $B_j\in \bM_{n_j}(\A_j)$ with $\|B_j\|=1$ such that
\[
\|P_{F}\rho^{(n_j)}(B_j)|_F\|<1
\]
for every irreducible $*$-representation $\rho:\rC^*(\A_j)\to B(\H_\rho)$ unitarily inequivalent to $\pi_j$ and every finite-dimensional subspace $F\subset \H^{(n_j)}_\rho$. Clearly, we see that $B_1\otimes B_2\in \bM_{n_1n_2}(\A_1\otimes \A_2)$ and $\|B_1\otimes B_2\|=1$.
It thus suffices to fix, for each $j=1,2,$  an irreducible $*$-representation $\sigma_j:\rC^*(\A_j)\to B(\K_j)$ such that $\sigma_1\otimes \sigma_2$ is unitarily inequivalent to $\pi_1\otimes \pi_2$, and to show that
\[
\|P_{F} (\sigma_1\otimes \sigma_2)^{(n_1n_2)}(B_1\otimes B_2)|_F\|<1
\]
for every finite-dimensional subspace $F\subset (\K_1\otimes \K_2)^{(n_1n_2)}\cong \K_1^{(n_1)}\otimes \K_2^{(n_2)}$. Using once again that $\rC^*(\A_1)$ is liminal, we see that $\K_1$ is finite-dimensional, and an elementary argument then reveals that it is sufficient to assume that $F=F_1\otimes F_2$ for some finite-dimensional subspaces $F_1\subset \K_1^{(n_1)}$ and $F_2\subset \K_2^{(n_2)}$.  
Since $\sigma_1\otimes \sigma_2$ is unitarily inequivalent to $\pi_1\otimes \pi_2$, either $\sigma_1$ is unitarily inequivalent to $\pi_1$, or $\sigma_2$ is unitarily inequivalent to $\pi_2$, so that  
\[
\|P_{F_1}\sigma_1^{(n_1)}(B_1)|_{F_1}\|  \|P_{F_2}\sigma_2^{(n_2)}(B_2)|_{F_2}\|
<1.
\]
Then,
\begin{align*}
&\|P_{F_1\otimes F_2}(\sigma_1\otimes \sigma_2)^{(n_1n_2)}(B_1\otimes B_2)|_{F_1\otimes F_2}\|\\
&= \|P_{F_1}\sigma_1^{(n_1)}(B_1)|_{F_1}\|  \|P_{F_2}\sigma_2^{(n_2)}(B_2)|_{F_2}\|\\
&<1.
\end{align*}
as desired.
 \end{proof}

A crucial part of the previous argument is the fact that any irreducible $*$-representation of $\rC^*(\A_1)\otimes \rC^*(\A_2)$ is unitarily equivalent to a tensor product of irreducible $*$-representations. By  \cite[Lemma 10]{wulfsohn1963}, this actually holds as long as one the algebras $\rC^*(\A_1), \rC^*(\A_2)$ is postliminal (or GCR) -- a weaker condition than liminality. However, in this more general context the rest of the proof does not carry over unaltered, as it is not clear then whether it suffices to consider finite-dimensional subspaces of the form $F_1\otimes F_2$.

Additionally, we mention that it is possible for $\rC^*(\A)$ to be liminal, while $\A$ does not have the Bishop property inside $\rC^*(\A)$ (Example \ref{E:nonBishop}). 

Using the previous theorem, we can now give more examples of operator algebras with the Bishop property.

\begin{corollary}\label{C:tensorBishop}
Let $X$ be a compact Hausdorff space and let $\fF$ be a regular, unitarily invariant, complete Pick space on the unit ball.  Let $\B\subset \rC(X)$ be a uniform algebra. Then, $\rA(\fF)\otimes \B$ has the Bishop property inside of $\fT(\fF)\otimes \rC(X)$.\end{corollary}
\begin{proof}
Recall that $\B$ has the Bishop property in $\rC(X)$ by Bishop's theorem. Likewise,  $\rA(\fF)$ has the Bishop property in $\fT(\fF)$ by Theorem \ref{T:AHBishop}. Therefore, the desired conclusion follows from Theorem \ref{T:tensorBishop}.
\end{proof}

\subsection{Minimality of the Choquet boundary}\label{SS:min}

Recall that if $X$ is a compact metric space and $\A\subset \rC(X)$ is a uniform algebra, then the Choquet boundary of $\A$ coincides with its set of peak points, by Bishop's theorem. As a consequence, the Choquet boundary is the smallest boundary for $\A$. Our purpose in this subsection is to determine to what extent a similar fact holds in the non-commutative context. 

Let $\A\subset B(\H)$ be a unital operator algebra. A set $\Delta$ of irreducible $*$-repre\-sentations of $\rC^*(\A)$ is a \emph{boundary} for $\A$ if for every integer $n\geq 1$ and for every $B\in \bM_n(\A)$ there is an element of $\Delta$, say $\pi:\rC^*(\A)\to B(\H_\pi)$, and a finite-dimensional subspace $F\subset \H_\pi^{(n)}$ such that 
\[
\|P_F \pi^{(n)}(B)|_F\|=\|B\|.
\]
It was observed in  \cite[Theorem 2.10]{CTh2020fdim} that any unital operator algebra admits a boundary consisting of boundary representations. As we will see shortly, this is more or less the generic situation.

We first record an elementary observation.

\begin{lemma}\label{L:peakandbdry}
Let $\A\subset B(\H)$ be a unital operator algebra and let $\pi$ be an irreducible $*$-representation of $\rC^*(\A)$. Then, the following statements are equivalent.
\begin{enumerate}[{\rm (i)}]
\item $\pi$ is a local $\A$-peak representation.
\item If $\Delta$ is a boundary for $\A$, then $\pi$ is unitarily equivalent to some element of $\Delta$.
\end{enumerate}
\end{lemma}
\begin{proof}
(i) $\Rightarrow$ (ii): There is $n\in \bN$ and $B\in \bM_n(\A)$ with $\|B\|=1$ such that, for every irreducible $*$-representation $\sigma:\rC^*(\A)\to B(\H_\sigma)$ unitarily inequivalent to $\pi$ and every finite-dimensional subspace $F\subset \H^{(n)}_\sigma$, we have that
\[
\|P_{F}\sigma^{(n)}(A)|_{F}\|<1.
\]
Now, let $\Delta$ be a boundary for $\A$. There is an element $\delta:\rC^*(\A)\to B(\H_\delta)$ in $\Delta$ along with a finite-dimensional subspace $G\subset \H^{(n)}_\delta$ such that \[\|P_{G}\delta^{(n)}(A)|_{G}\|=1.\] Clearly, this forces $\delta$ to be unitarily equivalent to $\pi$, as desired.

(ii) $\Rightarrow$ (i):  Let $\Delta$ be a set of irreducible $*$-representations of $\rC^*(\A)$ such that every irreducible $*$-representation is unitarily equivalent to an element of $\Delta$. Clearly, $\Delta$ is a boundary for $\rC^*(\A)$, and hence for $\A$. Let $\Delta_\pi\subset \Delta$ be the subset consisting of those elements unitarily inequivalent to $\pi$. By assumption, we see that $\Delta\setminus \Delta_\pi$ is not a boundary for $\A$, so there is  $n\geq 1$ and $B\in \bM_n(\A)$ such that for every irreducible $*$-representation $\sigma: \rC^*(\A)\to B(\H_\sigma)$  unitarily inequivalent to $\pi$ and every finite-dimensional subspace $F\subset \H_\sigma^{(n)}$, we have $\|B\|>\|P_F \sigma^{(n)}(B)|_F\|$. In other words, $\pi$ is a local $\A$-peak representation. 
\end{proof}

As a consequence, we show that for separable $\rC^*$-algebras, the notion of boundary is trivial.

\begin{corollary}\label{C:minbdryC*}
Let $\Delta$ be a set of irreducible $*$-representations of some separable unital $\rC^*$-algebra $\fT$. Then, $\Delta$ is a boundary for $\fT$ if and only if every irreducible $*$-representation of $\fT$ is unitarily equivalent to an element of $\Delta$.
\end{corollary}
\begin{proof}
Assume first that $\Delta$ is a boundary and that $\pi$ is an irreducible $*$-representation of $\fT$. Then, $\pi$ is a  local $\fT$-peak representation by Corollary \ref{C:C*peak}. Therefore, Lemma \ref{L:peakandbdry} implies that $\pi$ is unitarily equivalent to an element of $\Delta$. 
The converse is elementary. 
\end{proof}

Going back to the setting of a general unital operator algebra, we characterize the Bishop property by a certain minimality condition on the Choquet boundary.

\begin{theorem}\label{T:minbdry}
Let $\A\subset B(\H)$ be a unital operator algebra. Then, the following statements are equivalent.
\begin{enumerate}[{\rm (i)}]
\item The algebra $\A$ has the Bishop property inside of $\rC^*(\A)$.
\item A set $\Delta$ of irreducible $*$-representations of $\rC^*(\A)$ is a boundary for $\A$ if and only if any boundary representation for $\A$ is unitarily equivalent to some element of $\Delta$.
\end{enumerate}
\end{theorem}
\begin{proof}
(i) $\Rightarrow$ (ii):  A set $\Delta$ for which any boundary representation for $\A$ is unitarily equivalent to one of its elements  must be a boundary for $\A$, by virtue of \cite[Theorem 2.10]{CTh2020fdim}. Conversely, assume that $\Delta$ is a boundary for $\A$. Because $\A$ has the Bishop property inside of $\rC^*(\A)$,  every boundary representation is a local $\A$-peak representation, and hence is unitarily equivalent to an element of $\Delta$ by Lemma \ref{L:peakandbdry}. 

(ii) $\Rightarrow$ (i):  Let $\beta$ be a boundary representation for $\A$. By assumption, any boundary for $\A$ must contain an element unitarily equivalent to $\beta$. Lemma \ref{L:peakandbdry} then implies that $\beta$ is a local $\A$-peak representation.
\end{proof}

Using our earlier results on the Bishop property, we can now give examples of operators algebras that are not uniform algebras, and for which the Choquet boundary is a minimal boundary, thereby addressing our main motivating question.

\begin{corollary}\label{C:AHmin}
Let $\fF$ be a regular, unitarily invariant, complete Pick space on the unit ball. Let $\Delta$ be a set of irreducible $*$-representations of $\fT(\fF)$. Then, $\Delta$ is a boundary for $\rA(\fF)$ if and only if every boundary representation for $\rA(\fF)$ is unitarily equivalent to some element of $\Delta$. 
\end{corollary}
\begin{proof}
Combine Theorems \ref{T:AHBishop} and \ref{T:minbdry}.
\end{proof}

\begin{corollary}\label{C:tensormin}
Let $X$ be a compact Hausdorff space and let $\fF$ be a regular, unitarily invariant, complete Pick space on the unit ball.  Let $\B\subset \rC(X)$ be a uniform algebra. Let $\Delta$ be a set of irreducible $*$-representations of $\fT(\fF)\otimes \rC(X)$. Then, $\Delta$ is a boundary for $\rA(\fF)\otimes \B$ if and only if every boundary representation for $\rA(\fF)\otimes \B$ is unitarily equivalent to some element of $\Delta$. 
\end{corollary}
\begin{proof}
Combine Corollary \ref{C:tensorBishop} with Theorem \ref{T:minbdry}.
\end{proof}

\section{Boundaries for $\rC^*$-algebras}\label{S:bdryC*}

We saw in Corollary \ref{C:minbdryC*} that the previously introduced notion of a boundary is trivial for $\rC^*$-algebras. In this section, we study variations of this notion that carry non-trivial information.

Let $\fT$ be a $\rC^*$-algebra. Recall that the \emph{spectrum} of $\fT$, denoted by $\spec \fT$,  is the set of unitary equivalence classes of irreducible $*$-representations of $\fT$. Henceforth, we denote by $[\pi]$ the unitary equivalence class of an irreducible $*$-representation $\pi$. The natural association $[\pi]\mapsto \ker \pi$ induces a map from $\spec \fT$ onto the primitive ideal space $\Prim \fT$. Pulling back the Jacobson topology on $\Prim \fT$, we obtain a topology on $\spec \fT$ (see \cite[Chapter 3]{dixmier1977} for a detailed account on these topics). In particular, given $[\pi],[\sigma]\in \spec \fT$, we see that $[\pi]\in \ol{\{[\sigma]\}}$ if and only if $\ker \sigma\subset \ker \pi$.

Let $\Delta$ be a set of irreducible $*$-representations. We say that $\Delta$ is a \emph{maximizing boundary} for $\fT$ if for every $t\in \fT$ there exists $\delta\in \Delta$ with $\|\delta(t)\|=\|t\|$. These can be characterized as follows.

\begin{theorem}\label{T:maxbdryC*}
Let $\fT$ be a separable $\rC^*$-algebra. Then, the following statements are equivalent. 
\begin{enumerate}[{\rm (i)}]
\item The set  $\Delta$ is a maximizing boundary for $\fT$.
\item The kernel of any irreducible $*$-representation of $\fT$ contains the kernel of some element of $\Delta$. 
\item We have $\spec \fT=\bigcup_{\delta\in \Delta} \ol{\{[\delta]\}}$.
\end{enumerate}
\end{theorem} 
 \begin{proof}
 (ii) $\Leftrightarrow$ (iii): This follows immediately from the definition of the topology on $\spec \fT$.
 
 (ii) $\Rightarrow$ (i): Let $t\in \fT$. There always exists an irreducible $*$-representation $\pi$ of $\fT$ such that $\|\pi(t)\|=\|t\|$. By assumption, there is $\delta\in \Delta$ such that $\ker \delta\subset \ker \pi$, whence $\|\delta(t)\|\geq \|\pi(t)\|=\|t\|$. This shows that $\Delta$ is indeed a maximizing boundary for $\fT$.
 
  (i) $\Rightarrow$ (ii):  Let $\pi:\fT\to B(\H_\pi)$ be an irreducible $*$-representation and assume that $\ker \delta$ is not contained in $\ker \pi$ for any $\delta\in \Delta$. For each $\delta\in \Delta$, consider the non-zero $*$-representation $\pi_\delta=\pi|_{\ker \delta}:\ker \delta\to B(\H_\pi)$.  Since $\pi$ is irreducible,  we infer that so is $\pi_\delta$ \cite[Theorem 1.3.4]{arveson1976inv}. Let $\xi\in \H_\pi$ be a unit vector. By Kadison's transitivity theorem \cite[II.6.1.13]{blackadar2006}, there is $b_\delta\in \ker \delta$ with $\|b_\delta\|\leq 1$ such that $\pi_\delta(b_\delta)\xi=\xi$. 
 Since $\fT$ is separable, so is $\{b_\delta:\delta\in \Delta\}$. Correspondingly, let $\{c_n:n\in \bN\}$ be a countable dense subset of $\{b_\delta:\delta\in \Delta\}$ and define $a=\sum_{n=1}^\infty 2^{-n}c_n$. Clearly, $a\in \fT$, $\|a\|\leq 1$
and  $\pi(a)\xi=\xi$ so in fact $\|a\|=1$. 
  Finally, let $\delta\in \Delta$. There is $n_\delta\in \bN$ such that $\|c_{n_\delta}-b_\delta\|<1$ and thus $ \|\delta(c_{n_\delta})\|<1$ since $\delta(b_{\delta})=0$. We find
 \begin{align*}
\|\delta(a)\|&\leq \sum_{n\neq n_\delta}\frac{1}{2^n}+\frac{\|\delta(c_{n_\delta})\|}{2^{n_\delta}}<\sum_{n=1}^\infty \frac{1}{2^n}=1.
 \end{align*}
We conclude that $\Delta$ is not a maximizing boundary for $\fT$.
 \end{proof}

The previous result provides a justification for our choice of terminology. 
Motivated by the classical setting, the notion of a maximizing boundary may appear as a perfectly natural candidate for what we actually called a boundary for an operator algebra in Subsection \ref{SS:min}. However, Theorem \ref{T:maxbdryC*} shows that, even for $\rC^*$-algebras, some spatial norm attainment information must be encoded in order to ensure that a boundary truly ``saturates'' the set of all irreducible $*$-representations.

 As an application of Theorem \ref{T:maxbdryC*}, we can give a different proof of a recent result of Courtney and Shulman \cite[Theorem 4.4]{CS2019}.
 
 \begin{corollary}\label{C:CS}
  Let $\fT$ be a $\rC^*$-algebra. Then, the following statements are equivalent.
  \begin{enumerate}[{\rm (i)}]

\item There is a maximizing boundary for $\fT$ consisting of finite-dimensional $*$-representations of $\fT$.

\item Every irreducible $*$-representation of $\fT$ is finite-dimensional.
\end{enumerate}
 \end{corollary}
 \begin{proof}
(ii) $\Rightarrow$ (i): This is clear.

(i) $\Rightarrow$ (ii): It suffices to deal with the case where $\fT$ is separable, upon applying  \cite[Theorem 1]{batty1984}. Let $\Delta$ be a maximizing boundary for $\fT$ consisting of finite-dimensional $*$-representations. Let $\pi$ be an irreducible $*$-representation of $\fT$. By Theorem \ref{T:maxbdryC*}, there is $\delta\in \Delta$ such that $\ker \delta\subset \ker \pi$. In particular, this shows that $\pi(\fT)$ is a quotient of the finite-dimensional $\rC^*$-algebra $\delta(\fT)$. Consequently, $\pi$ is finite-dimensional.
 \end{proof}
 
The previous argument is different from the original proof of \cite[Theorem 4.4]{CS2019}, which makes use of so-called AF mapping telescopes.

As before, let $\fT$ be a $\rC^*$-algebra and let $\Delta$ be a collection of irreducible $*$-representations of $\fT$. To motivate the next development, we remark here that $\Delta$ is a maximizing boundary for $\fT$ precisely when for every $t\in \fT$ and every irreducible $*$-representation $\pi$ we can find $\delta\in \Delta$ such that 
\[
\|\delta(t)\|\geq \|\pi(t)\|.
\]
We say that $\Delta$ is a \emph{minimizing boundary} for $\fT$ if
for every $t\in \fT$ and every irreducible $*$-representation $\pi$ we can find $\delta\in \Delta$ such that 
\[
\|\delta(t)\|\leq \|\pi(t)\|.
\]
 Minimizing boundaries can be characterized in a manner similar to maximizing ones.

\begin{theorem}\label{T:minbdryC*}
Let $\Delta$ be a collection of irreducible $*$-representations of some separable $\rC^*$-algebra $\fT$.  Then, the following statements are equivalent. 
\begin{enumerate}[\rm (i)]
\item The set $\Delta$ is a minimizing boundary for $\fT$.
\item The kernel of any irreducible $*$-representation of $\fT$ is contained in the kernel of some element of $\Delta$. 
\end{enumerate}
\end{theorem}

\begin{proof}
{\rm (ii)}$\Rightarrow${\rm (i)}: Fix $t\in \fT$ and an irreducible $*$-representation $\pi$. By assumption, there is $\delta\in\Delta$ such that $ \ker\pi\subset\ker\delta$. In particular, $\lVert\delta(t)\rVert\leq \lVert \pi(t)\rVert$. 

{\rm (i)}$\Rightarrow${\rm (ii)}: Let $\pi:\fT\to B(\H_\pi)$ be an irreducible $*$-representation and let $\xi\in\H_\pi$ be a unit vector. For each $\delta:\fT\to B(\H_\delta)$ in $\Delta$, let $\eta_\delta\in\H_\delta$ be a unit vector. Choose a contractive linear operator $X_\delta:\H_\pi\rightarrow\H_\delta$ satisfying $X_\delta\xi = \eta_\delta.$ Note then that $X_\delta^*\eta_\delta = \xi$ by the Cauchy--Schwarz inequality.

Assume that $\ker\delta$ does not contain $\ker\pi$ for any $\delta\in\Delta$. Then, $\delta\mid_{\ker\pi}$ is an irreducible $*$-representation \cite[Theorem 1.3.4]{arveson1976inv}. By Kadison's transitivity theorem, there is a self-adjoint element $b_\delta\in\ker\pi$ with $\|b_\delta\|\leq 1$ such that $\delta(b_\delta)\eta_\delta = \eta_\delta$.  Upon replacing $b_\delta$ by $b_\delta^2$ if necessary, we may assume that $0\leq b_\delta\leq I$. 
Since $\fT$ is separable, so is $\{ b_\delta:\delta\in\Delta\}$. Correspondingly, let $\{ c_n: n\in\bN\}$ be a countable dense subset of $\{ b_\delta:\delta\in\Delta\}$ and define $a = \sum_{n=1}^\infty 2^{-n} c_n.$ Clearly, we have that $a\in \ker\pi$. 

We now fix $\delta\in \Delta$ and claim that $\delta(a)\neq 0$. To see this, observe first hat \[ X_\delta^*\delta(b_\delta)X_\delta\xi = \xi\]and so $\lVert X_\delta^*\delta(b_\delta)X_\delta\rVert=1$. Let $0<\varepsilon<1$. Then, there is $n_\delta\in\bN$ such that $\lVert b_\delta- c_{n_\delta}\rVert<\varepsilon$. Therefore, \[ \lVert X_\delta^*\delta(c_{n_\delta})X_\delta\xi\rVert \geq \lVert X_\delta^*\delta(b_\delta)X_\delta\xi\rVert -\eps= 1-\varepsilon.\]
Since each $c_n$ is a positive element, we see that $a\geq 2^{-m} c_m$ for every $m\geq 1$. Then,
\[
X_\delta^*\delta(a)X_\delta\geq \frac{1}{2^{n_\delta}}X_\delta^*\delta(c_{n_\delta})X_\delta\geq 0
\] 
and  \[ \lVert X_\delta^*\delta(a)X_\delta\lVert \geq \frac{1}{2^{n_\delta}}\lVert X_\delta^*\delta(c_m)X_\delta\rVert \geq \frac{1}{2^{n_\delta}}(1-\varepsilon)>0.\] 
This establishes the claim that $\delta(a)\neq 0$. Consequently, $\Delta$ is not a minimizing boundary for $\fT$, since $\pi(a)=0$.
\end{proof}

\bibliography{minbdry}
\bibliographystyle{plain}

\end{document}